\documentclass[reqno]{amsart}

\usepackage{url}
\usepackage{a4wide}

\usepackage[utf8]{inputenc}
\usepackage[T1]{fontenc}
\usepackage{lmodern}

\usepackage{amsmath}
\usepackage{amssymb}
\usepackage{amsfonts}
\usepackage{amscd}
\usepackage{amsthm}
\usepackage{caption}
\usepackage{subcaption}

\usepackage{graphicx}
\usepackage[english]{babel}

\usepackage{wrapfig}

\usepackage[all]{xy}

\usepackage[ruled,vlined,boxed,linesnumbered,norelsize]{algorithm2e}

\allowdisplaybreaks

\theoremstyle{plain}
\newtheorem{theorem}{Theorem}[section]

\newtheorem{proposition}[theorem]{Proposition}
\newtheorem{lemma}[theorem]{Lemma}

\theoremstyle{definition}

\theoremstyle{remark}

\numberwithin{equation}{section}

\DeclareMathOperator{\GL}{GL}

\DeclareMathOperator{\PGL}{PGL}

\DeclareMathOperator{\Tr}{Tr}

\newcommand{\mybar}[1]{
  \mathchoice
  {#1\llap{$\overline{\phantom{\displaystyle\rm#1}}$}}
  {#1\llap{$\overline{\phantom{\textstyle\rm#1}}$}}
  {#1\llap{$\overline{\phantom{\scriptstyle\rm#1}}$}}
  {#1\llap{$\overline{\phantom{\scriptscriptstyle\rm#1}}$}}
} 
\renewcommand{\bar}{\mybar}

\def\F{\mathbf{F}}

\def\N{\mathbf{N}}

\def\P{\mathbf{P}}

\newcommand{\rR}{\mathcal{R}}
\newcommand{\sS}{\mathcal{S}}
\newcommand{\tT}{\mathcal{T}}

\usepackage{mdwlist}

\usepackage{color}

\def\phi{\varphi}

\title{Good recursive towers over prime fields exist}

\author[Bassa]{Alp Bassa}
\address{%
Bo\u gazi\c ci University,
Deparment of Mathematics,
34342 Bebek, Istanbul,
Turkey
}
\email{alp.bassa@boun.edu.tr}

\author[Ritzenthaler]{Christophe Ritzenthaler}
\address{%
  Institut de recherche math\'ematique de Rennes, %
  Universit\'e de Rennes 1, %
  Campus de Beaulieu, %
  35042 Rennes, %
  France. %
}
\email{christophe.ritzenthaler@univ-rennes1.fr }

\thanks{The authors acknowledge support by the PHC Bosphorus 39652NB - T\"UB\.ITAK 117F274 and thank the Nesin Mathematical Village for its inspiring environment.}

\date{\today} 
\subjclass[2010]{11G20, 11T71, 14H25, 14G05, 14G15}
\keywords{recursive tower ; explicit ; prime field ; correspondences ; Singer subgroup}

\begin{document}

\maketitle

\begin{abstract}
We give a construction and equations for good recursive towers over any finite field $\F_q$ with $q \ne 2$ and $3$.
\end{abstract}

\section{Introduction}

Let $\F_q$ be a finite field with $q=p^n$ elements, where $p$ is a prime and $n \geq 1$ an integer. A central quantity in the theory of curves over finite fields of large genus is the Ihara constant $A(q)$, which is defined as 
$$A(q):=\limsup_{g \to \infty} \frac{N_q(g)}{g},$$
where $N_q(g) =\max \#X(\F_q)$ with $X$ running over all  smooth projective absolutely irreducible curves defined over $\F_q$ of genus $g(X)=g > 0$. 
Drinfel'd--Vl\u{a}du\c{t} \cite{drinvla} obtained the inequality 
\begin{equation}\label{eq:drinvla}
A(q)\leq \sqrt{q}-1,
\end{equation}
which is still the only known upper bound. 

When the exponent $n$ is even,  Ihara \cite{ihara} used reductions of Shimura curves to show that equality holds in \eqref{eq:drinvla}. This result was also obtained by Tsfasman--Vl\u{a}du\c{t}--Zink \cite{tvz} for $n=2$ and $4$  using reductions of elliptic modular curves.  For any $q$, using class field theory, Serre \cite{serre} showed  that $A(q)>c\log(q)$ for some constant $c>0$ independent of $q$ (one can take for instance $c=\frac{1}{96}$ \cite[Theorem 5.2.9]{niexin}). In particular $A(q)>0$ for all $q$. The exact value of $A(q)$ is however unknown when $q$ is not a square.

Garcia and Stichtenoth \cite{garsti} marked a major turning-point in the theory by introducing the notion of recursive towers $\tT=(X_m)_{m \in \N}$. These towers are described by two morphisms $f,g : C_0 \rightrightarrows C_{-1}$ where $C_0$ and $C_{-1}$ are curves defined over $\F_q$. This defines recursively $(C_m)_{m \geq 0}$ by the fiber product
\begin{center}
\leavevmode
\xymatrix{ C_{m+1} \ar[r] \ar[d]_{\pi_m} & C_0 \ar[d]^g \\
C_m \ar[r]_{f_m} & C_{-1} 
}
\end{center}
where 
$f_m(P_0,\ldots,P_m)=f(P_m)$  and $\pi_m(P_0,\ldots,P_{m+1}) = (P_0,\ldots,P_m)$. In other terms one has
$$C_m = \{(P_0,\ldots,P_m) \in C_0^{m+1} \, : \; g(P_i)=f(P_{i-1}), 1 \leq i \leq  m\}.$$
 One then considers the normalization $X_m$ of $C_m$ and we will still denote by $\pi_m : X_{m+1} \to X_m$ the induced cover.  Although the curves $X_m$ are smooth, it is not automatic that they are absolutely irreducible or that their genus goes to infinity. If it is so, $\tT=(X_m)_{m \in \N}$ with the morphisms $\pi_m$ is called a \emph{tower}. It is a \emph{good tower} if the \emph{limit of the tower} 
 $$\lambda(\tT):=\limsup \frac{\#X_m(\F_q)}{g(X_m)}$$
is positive and  an \emph{optimal tower} if it reaches the Drinfel'd--Vl\u{a}du\c{t} bound. 

For $n$  even, Garcia and Stichtenoth exhibited explicit examples of optimal towers. Compared to constructions using class field towers or modular curves, recursive towers have the advantage of being more elementary and explicit, which is crucial for potential applications in coding theory and cryptography. This approach resulted subsequently in the discovery of many good  recursive towers when $n=2$ (see \cite{garsti2, garsti3}, among others), $n=3$ (see \cite{cubic1, cubic2}, among others) and $n>1$ (see \cite{BBGS}). The last one gives, under a unified construction, the best known lower bounds. The morphisms $f$ and $g$ resulting in good towers seem to be very special and in fact almost all have been proven to have a modular interpretation of various types (see \cite{elkies1, elkies2, ABBcubic} among others).

Still, when $n=1$, among the several approaches mentioned above, only class field theory has produced positive lower bounds for the Ihara constant $A(q)$. Despite several decades of attempts, no good modular towers or recursive towers were obtained and their existence has even been questioned. 

The present article breaks this barrier by exhibiting a good recursive tower $\tT=(X_m)_{m \in \N}$ over any field $\F_q$ with $q>3$ such that $\lambda(\tT) \geq \frac{2}{q-2}$. The structure of the tower is surprisingly simple. Starting from a good choice of cover $f : C_0 \to C_{-1}$, one constructs the morphism $g$ by choosing carefully an automorphism $\psi$ (resp. $\phi$) of $C_0$ (resp. $C_{-1})$ and defines $g=\phi \circ f \circ \psi$. More precisely, let us denote by $\rR$ a subset of $C_0(\bar{\F}_q)$ containing the ramification points of $f$ and $\sS \subset C_0(\F_q)$ the inverse image under $f$ of a set of totally split points of $f$. Let us assume that there exist an automorphism $\psi$ of $C_0$ (resp. $\phi$ of $C_{-1}$) preserving $\rR$ and $\sS$ (resp. $f(\rR)$ and $f(\sS)$). If we can ensure that $f,g$ define a tower $\tT=(X_m)_{m \in \N}$, then $\rR$ (resp. $\sS$) contains  the ramification locus (resp. splitting locus) of the tower  as defined in \cite{Stichtenoth}.  The general formula of \cite[Corollary 7.2.11.]{Stichtenoth}  then gives that $\lambda(\tT) \geq \frac{2 \#\sS}{\# \rR-2}$.

 We choose a particular cyclic Galois cover $f : \P^1 \to \P^1$ of degree $q+1$. The ramification points $Q,\bar{Q}$ of $f$ are defined over $\F_{q^2}$ whereas the point at infinity is totally split, i.e. $f^{-1}(\infty)=\P^1(\F_q)$. We let $\rR=\P^1(\F_{q^2}) \setminus \P^1(\F_q)$ and $\sS=\P^1(\F_q)$ and prove  that there exist automorphisms $\psi$ and $\phi$ preserving these sets and their image by $f$.
 Moreover, when $q>3$, we can impose that $g(Q)=Q$ and $Q,\bar{Q}$ are not ramification points of $g$. This implies that the curves $X_m$ are absolutely irreducible and of growing genus hence $\tT$ is a good tower. 

Our construction does not work over $\F_2$ and $\F_3$ and does not improve any of the known lower bounds on $A(q)$. It does, however, provide an elementary proof for $A(q)>0$ in case $q>3$. The simple framework we suggest makes it tempting to improve the limit by changing the cover $f : C_{0} \to C_{-1}$. Using the map $\mu$ in the proof of Lemma~\ref{lem:phi}, it is also possible to recover the tower as a twist of a good tower over a quadratic field extension. It might be possible to generalize this principle to other good towers. A further interesting question is to give (if possible) a modular interpretation for this tower. Our construction gives a first example of recursive towers over prime fields and we hope that it will stimulate the quest for better ones.

\section{Main result}
\label{sec:abs}

Let $q=p^n$ be a prime power and for $x \in \F_{q^2}$ denote $\bar{x}$ its Galois conjugate over $\F_q$. Let  $Q=(\theta:1) \in \P^1(\F_{q^2})\backslash \P^1(\F_{q})$. The automorphism group of $\P^1$ over $\F_q$ is $\PGL_2(\F_q)$ and the isotropy group $G$ of $Q$ is a cyclic subgroup of $\PGL_2(\F_q)$ of order $q+1$. The conjugate point $\bar{Q}=(\bar{\theta}:1)$ has the same isotropy group $G$ and  all cyclic subgroups of $\PGL_2(\F_q)$ of order $q+1$ appear as isotropy subgroups of such $Q$. Because of the transitive action of $\PGL_2(\F_q)$ on $\P^1(\F_{q^2})\backslash \P^1(\F_{q})$, they are all conjugated (see for instance \cite[Th.2]{ValentiniMadan}). Seen as a subgroup of $\GL_2(\F_q)$, these subgroups are called \emph{Singer subgroups} and have been well studied.

The group $G$ defines a tame Galois cover $f : \P^1 \to \P^1$ of degree $q+1$, where $Q$ and $\bar{Q}$ are totally ramified and, by Riemann-Hurwitz formula, the only ramified points. As a consequence, since the order of $G$ is $q+1$, it acts transitively on $\P^1(\F_q)$.
We choose coordinates $(x:z)$ on $\P^1$ such that $f(Q)=Q$, $f(\bar{Q})=\bar{Q}$ and $f(\infty)=\infty$ where $\infty=(0:1)$. We then have that $f^{-1}(\infty) = \P^1(\F_p)$. Therefore $\infty$ is totally split whereas $Q$ and $\bar{Q}$ are the two branched points.
As mentioned in the introduction, we denote  $\rR=\P^1(\F_{q^2}) \setminus \P^1(\F_q)$ and $\sS=\P^1(\F_q)$. 

\begin{lemma} \label{lem:phi}
We have $f(\rR)=\{(\gamma:1)\, :\gamma \in \F_{q^2} \textrm{ with }\Tr_{\F_{q^2}/\F_q}(\gamma)=\Tr_{\F_{q^2}/\F_q}(\theta)\}$.
\end{lemma}
\begin{proof}
Let $\mu : x\mapsto (x-\theta)/(x-\bar{\theta})$ and consider $\tilde{f}=\mu \circ f \circ \mu^{-1} : \P^1 \to \P^1$. It defines a Galois cover of degree $q+1$ with fixed points $(0:1)$ and $\infty$. Hence $\tilde{f}$ corresponds to the rational map  $x \mapsto c N(x)$, with $N(x)=x^{q+1}$ for some constant $c \in \F_{q^2}$.  Actually $c=1$, since
$$\tilde{f}((\theta/\bar{\theta}:1)) = (c:1) = \mu\circ f((0:1)) = \mu(\infty) =(1:1).$$
If $u \in \F_q$ then ${(u-\theta)}/{(u-\bar{\theta})}={(u-\theta)}/{(\bar{u-\theta})}$ hence $N \circ \mu(P)= (1:1)$ for all $P \in \P^1(\F_q)$. The surjectivity of the norm map from $\F_{q^2}$ to $\F_q$ implies that $$N(\mu(\rR))=\P^1(\F_q) \setminus \{(1:1)\} .$$
It is immediate to verify that 
\[f(\rR)=\mu^{-1}\left(\P^1(\F_q) \setminus \{(1:1)\}\right)=\{(\gamma:1)\, :\gamma \in \F_{q^2} \textrm{ with }\Tr_{\F_{q^2}/\F_q}(\gamma)=\Tr_{\F_{q^2}/\F_q}(\theta)\}.\tag*{\qedhere}\]
\end{proof}

 Looking for automorphisms $\phi,\psi$ of $\P^1$  such that $\psi$ preserves $\rR, \sS$ and $\phi$ their image under $f$, we see that it is necessary and sufficient for $\psi$ to be defined over $\F_q$ to preserve $\rR$ and $\sS$ and hence $\phi$ is also defined over $\F_q$ in order for $g$ to be defined over $\F_q$. 
Moreover, since $f(\sS)=\infty$, we need $\phi$ to be of the form $(x:z) \mapsto (c x+ d z:z)$. 

\begin{proposition} \label{prop:phi}
An automorphism $\phi : \P^1 \to \P^1$  of the form $(x:z) \mapsto (c x + d z:z)$ and defined over $\F_q$ preserves $f(\rR)$ if and only if 
$(1-c) a = 2 d$ with $c \ne 0$ and $a=\Tr_{\F_{q^2}/\F_q}(\theta)$.
\end{proposition}
\begin{proof}
An automorphism $(x:z) \mapsto (c x + d z:z)$ preserves $$f(\rR)=\{(\gamma:1)\, :\gamma \in \F_{q^2} \textrm{ with }\Tr_{\F_{q^2}/\F_q}(\gamma)=a\}$$ if and only if $\Tr_{\F_{q^2}/\F_q}(c\gamma+d)=a$ i.e. $(1-c)a=2d$.
\end{proof}

 \begin{theorem}
 Assume that $q>3$. There exists an explicit recursive tower $\tT=(X_m)_{m \in \N}$ over $\F_q$ with limit $$\lambda(\tT)\geq \frac{2}{q-2}.$$ 
 \end{theorem}
\begin{proof}
Let $g=\phi \circ f \circ \psi$ for a choice of an automorphism $\psi$ (resp. $\phi$) defined over $\F_q$ and preserving $\rR$  and $\sS$ (resp. $f(\rR)$ and $f(\sS)$). Consider the correspondence $f,g : \P^1 \rightrightarrows \P^1$  and the curves $(X_m)_{m \in \N}$ they define. For $\tT=(X_m)_{m \in \N}$ to be a tower, it is enough to show that all the separable morphisms $\pi_{m+1}:X_{m+1}\to X_m$ are ramified. To ensure this, it is enough that $g(Q)=Q$ and that $Q$ is not a ramification point for $g$, i.e. $\phi(Q) \notin \{Q, \bar{Q}\}$. For $q>3$, there always exists $c,d \in \F_q$ with $c\neq 0$ and $(1-c) a=2d$ (as in Proposition~\ref{prop:phi}) such that $\phi(x)=cx+d$ does not map $Q$ on itself ($c=1$, $d=0$) or on $\bar{Q}$ ($c=-1$, $d=a$). Let us pick such $(c,d)$ and corresponding $\phi$.  Let us now consider $S=\phi^{-1}(Q)$. Since $\phi$ stabilizes $f(\rR)$ and $Q \in f(\rR)$, we see that the points of $f^{-1}(S) =\{T_1,\bar{T_1},\ldots, T_{(q+1)/2},\bar{T}_{(q+1)/2}\}$ are in $\rR$. One then picks an automorphism $\psi$ defined over $\F_q$ such that $\psi(Q)=T_i$ for $1 \leq i \leq \frac{q+1}{2}$. For such a choice, one has $g(Q)=Q$. The general formula of \cite[Corollary 7.2.11.]{Stichtenoth}  then gives that 
\[\lambda(\tT) \geq \frac{2 \#\sS}{\# \rR-2} = \frac{2}{q-2}. \tag*{\qedhere}\]
\end{proof}

Lastly, we derive in odd characteristic explicit equations for the maps $f$ and $g$ above. Let $\chi(X)=X^2-a X+b$ be the minimal polynomial of $\theta$. 
Requiring that $f(Q)=Q, f(\bar{Q})=\bar{Q}$ and $f(\infty)=\infty$ with $Q$ and $\bar{Q}$ totally ramified and $\infty$ totally split implies that 
$$f(x) = \frac{x^{q+1} -a x +b}{x^q-x}.$$

Let us now consider one of the $\psi$ from above and denote $R=(\rho:1)=\psi^{-1}(Q) \notin \{Q,\bar{Q}\}$ and $(\nu:1)=g(R)=\phi(Q)$. Let $X^2-t X+n$ be the minimal polynomial of $\rho$. The fact that  $g$ is defined over $\F_q$, $R$ is totally ramified for $g$ and $g^{-1}(\infty)=\P^1(\F_q)$ means that
 $$g(x)=\frac{x^{q+1}+c_q x^q+c_1 x+c_0}{c (x^q-x)}$$
where 
$$c_0=n,\quad c_q- c \nu= -\rho, \quad c_1 + c \nu = - \bar{\rho}$$
from which we derive $c_q+c_1=-t$ and $c_q-c_1= c \Tr_{\F_{q^2}/\F_q}(\nu)$.
Imposing that $g(Q)=Q$ we get that 
$$c=\frac{2b+2 n+t a}{4b-a^2} \ne 0, \quad c_q-c_1=c a.$$
Since $\Tr_{\F_{q^2}/\F_q}(\nu)=(c_q-c_1)/c=a=\Tr_{\F_{q^2}/\F_q}(\theta)$, the map $\phi$ satisfies the condition of Proposition~\ref{prop:phi} and we can choose $\rho \in \F_{q^2} \setminus \F_q$ arbitrarily as long as $\rho \ne \theta,\bar{\theta}$ and $2b+2t+ta \ne 0$.
Taking conveniently $a=0$ and $t=0$ and $-n, -b \in \F_q^{\times}$ two non-squares with $n \ne \pm b$ (hence $q>5$), we get for instance
$$f(x) = \frac{x^{q+1}+b}{x^q-x} \quad  \textrm{and} \quad g(x)= \frac{2b(x^{q+1}+n)}{(b+n) (x^q-x)}.$$

For $q=5$ one can check that with $\theta$ a root of $\chi(X)=X^2+ X+2$, the automorphisms $\phi(x)=2x+3$ and $\psi(x)=1/x$ satisfy the requirements of the theorem. Hence we obtain as possible maps $f(x) = (x^6+x+2)/(x^5-x)$ and $g(x)= (x^6+x^5+2x+3)/(x^5-x).$

\end{document}